\documentclass[draft]{amsart}
\usepackage{amssymb}
\usepackage{braket}
\usepackage[all]{xy}
\usepackage{mathrsfs}
\usepackage{braket}

\newtheorem{thm}{Theorem}[section]
\newtheorem{cor}[thm]{Corollary}
\newtheorem{prop}[thm]{Proposition}
\newtheorem{lem}[thm]{Lemma}

\theoremstyle{definition}
\newtheorem{dfn}[thm]{Definition}
\newtheorem{ex}[thm]{Example}
\newtheorem{fact}[thm]{Fact}

\theoremstyle{remark}
\newtheorem{rem}[thm]{Remark}
\newtheorem{caution}[thm]{Caution}

\newcommand{\Ob}{\mathrm{Ob}}
\newcommand{\ppr}{^{\prime}}

\newcommand{\Mon}{\mathit{Mon}}
\newcommand{\Ab}{\mathit{Ab}}

\newcommand{\Ring}{\mathit{Ring}}

\newcommand{\Sett}{\mathit{Set}}

\newcommand{\TamG}{\mathit{Tam}(G)}

\newcommand{\MackG}{\mathit{Mack}(G)}
\newcommand{\SMackG}{\mathit{SMack}(G)}

\newcommand{\Gs}{{}_G\mathit{set}}

\newcommand{\I}{\mathscr{I}}

\newcommand{\pt}{\mathrm{pt}}
\newcommand{\id}{\mathrm{id}}

\newcommand{\Sl}{\mathscr{S}}
\newcommand{\Ll}{\mathscr{L}}
\newcommand{\Ul}{\mathscr{U}}

\newcommand{\Z}{\mathfrak{Z}}
\newcommand{\MGs}{M\text{-}\Gs}

\numberwithin{equation}{section}

\begin{document}

\title[Fractions of semi-Mackey and Tambara functors]{On the fractions of semi-Mackey and Tambara functors}

\author{Hiroyuki NAKAOKA}
\address{Department of Mathematics and Computer Science, Kagoshima University, 1-21-35 Korimoto, Kagoshima, 890-0065 Japan}

\email{nakaoka@sci.kagoshima-u.ac.jp}

\thanks{The author wishes to thank Professor Serge Bouc and Professor Radu Stancu for the stimulating arguments and their useful comments and advices}
\thanks{The author wishes to thank Professor Fumihito Oda for his suggestions and useful comments}
\thanks{Supported by JSPS Grant-in-Aid for Young Scientists (B) 22740005}

\begin{abstract}
For a finite group $G$, a semi-Mackey (resp. Tambara) functor is regarded as a $G$-bivariant analog of a commutative monoid (resp. ring).
As such, some naive algebraic constructions are generalized to this $G$-bivariant setting.
In this article, as a $G$-bivariant analog of the fraction of a ring, we consider {\it fraction} of a Tambara (and a semi-Mackey) functor, by a multiplicative semi-Mackey subfunctor.

\end{abstract}

\maketitle

\tableofcontents

\section{Introduction and Preliminaries}

For a finite group $G$, a semi-Mackey functor (resp. a Tambara functor) is regarded as a $G$-bivariant analog of a commutative monoid (resp. ring), as seen in \cite{Yoshida}.
As such, some naive algebraic constructions are generalized to this $G$-bivariant setting. For example an analog of ideal theory was considered in \cite{N_IdealTam}, and an analog of monoid-ring construction was considered in \cite{N_TamMack}.

In the ordinary ring theory, {\it fraction} is another well-established construction.
If we are given a multiplicatively closed subset $S$ of a ring $R$, then there associated a ring $S^{-1}R$ and a natural ring homomorphism $\ell_S\colon R\rightarrow S^{-1}R$ satisfying some universality. Similarly for monoids.

As a $G$-bivariant analog of this, we consider {\it fraction} of a Tambara (and a semi-Mackey) functor, by a multiplicative semi-Mackey subfunctor.

\medskip

In this article, a monoid is always assumed to be unitary and commutative. Similarly a ring is assumed to be commutative, with an additive unit $0$ and a multiplicative unit $1$.
We denote the category of monoids by $\Mon$, the category of rings by $\Ring$, and the category of abelian groups by $\Ab$.
A monoid homomorphism preserves units, and a ring homomorphism preserves $0$ and $1$.

We always assume that a multiplicatively closed subset $S\subseteq R$ contains $1$. Thus a multiplicatively closed subset is nothing other than a submonoid of $R^{\mu}$, where $R^{\mu}$ denotes the underlying multiplicative monoid of $R$.
For any submonoid $S$ of a monoid $M$, its {\it saturation} $\widetilde{S}$ is defined by
\[ \widetilde{S}=\{ x\in M\mid ax=s\ \text{for some}\ a\in M, s\in S \}. \]
Then $\widetilde{S}\subseteq M$ is again a submonoid. 
$S$ is called {\it saturated} if it satisfies $S=\widetilde{S}$.

Remark also that if $M$ is a $G$-monoid and $S\subseteq M$ is $G$-invariant, its saturation $\widetilde{S}$ is also $G$-invariant.


\medskip

Throughout this article, we use the same basic notations as in \cite{N_IdealTam}. We fix a finite group $G$, whose unit element is denoted by $e$. Abbreviately we denote the trivial subgroup of $G$ by $e$, instead of $\{ e\}$.
$H\le G$ means $H$ is a subgroup of $G$.
$\Gs$ denotes the category of finite $G$-sets and $G$-equivariant maps.
The order of $H$ is denoted by $|H|$, and the index of $K$ in $H$ is denoted by $|H:K|$, for any $K\le H\le G$.

For any category $\mathscr{C}$ and any pair of objects $X$ and $Y$ in $\mathscr{C}$, the set of morphisms from $X$ to $Y$ in $\mathscr{C}$ is denoted by $\mathscr{C}(X,Y)$. 

\medskip

\section{Fraction of a semi-Mackey functor}

Before constructing a fraction of a Tambara functor, we introduce the fraction of a semi-Mackey functor.
First, we briefly recall the definition of a (semi-)Mackey functor.
Although a (semi-)Mackey functor seems to be recognized well enough, we add this section for the sake of self-containedness and to fix the notations.

\begin{dfn}\label{DefSemiMackFtr}
A {\it semi-Mackey functor} $M$ {\it on} $G$ is a pair $M=(M^{\ast},M_{\ast})$ of a covariant functor
\[ M_{\ast}\colon\Gs\rightarrow\Sett, \]
and a contravariant functor
\[ M^{\ast}\colon\Gs\rightarrow\Sett \]
which satisfies the following. Here $\Sett$ denotes the category of sets.
\begin{enumerate}
\item For each object $X\in\Ob(\Gs)$, we have $M_{\ast}(X)=M^{\ast}(X)$. We denote this simply by $M(X)$.
\item For any pair $X,Y\in\Ob(\Gs)$, if we denote the inclusions into $X\amalg Y$ by $\iota_X\colon X\hookrightarrow X\amalg Y$ and $\iota_Y\colon Y\hookrightarrow X\amalg Y$, then
\[ (M^{\ast}(\iota_X),M^{\ast}(\iota_Y))\colon M(X\amalg Y)\rightarrow M(X)\times M(Y) \]
becomes an isomorphism.
\item (Mackey condition)
If we are given a pullback diagram
\[
\xy
(-7,6)*+{X\ppr}="0";
(7,6)*+{X}="2";
(-7,-6)*+{Y\ppr}="4";
(7,-6)*+{Y}="6";
(0,0)*+{\square}="8";
{\ar^{\xi} "0";"2"};
{\ar_{f\ppr} "0";"4"};
{\ar^{f} "2";"6"};
{\ar_{\eta} "4";"6"};
\endxy
\]
in $\Gs$, then
\[
\xy
(-12,7)*+{M(X\ppr)}="0";
(12,7)*+{M(X)}="2";
(-12,-7)*+{M(Y\ppr)}="4";
(12,-7)*+{M(Y)}="6";
{\ar^{M_{\ast}(\xi)} "0";"2"};
{\ar^{M^{\ast}(f\ppr)} "4";"0"};
{\ar_{M^{\ast}(f)} "6";"2"};
{\ar_{M_{\ast}(\eta)} "4";"6"};
{\ar@{}|\circlearrowright "0";"6"};
\endxy
\]
is commutative.
\end{enumerate}

If $M$ is a semi-Mackey functor, then $M(X)$ becomes a monoid for each $X\in\Ob(\Gs)$, and $M^{\ast}$, $M_{\ast}$ become monoid-valued functors $\Gs\rightarrow\Mon$.
Those $M^{\ast}(f), M_{\ast}(f)$ for morphisms $f$ in $\Gs$ are called {\it structure morphisms} of $M$.
$M^{\ast}(f),M_{\ast}(f)$ are often abbreviated to $f^{\ast},f_{\ast}$.

A {\it morphism} of semi-Mackey functors $\vartheta\colon M\rightarrow N$ is a family of monoid homomorphisms
\[  \vartheta=\{\vartheta_X\colon M(X)\rightarrow N(X) \}_{X\in\Ob(\Gs)}, \]
natural with respect to the contravariant and the covariant parts. We denote the category of semi-Mackey functors by $\SMackG$.

If $M$ is a semi-Mackey functor on $G$, a {\it semi-Mackey subfunctor} $\Sl\subseteq M$ is a family of submonoids $\{ \Sl(X)\subseteq M(X) \}_{X\in\Ob(\Gs)}$, satisfying
\[ f^{\ast}(\Sl(Y))\subseteq\Sl(X),\quad f_{\ast}(\Sl(X))\subseteq\Sl(Y) \]
for any $f\in\Gs(X,Y)$. Then $\Sl$ itself becomes a semi-Mackey functor, and this is nothing other than a subobject in $\SMackG$.

A semi-Mackey functor $M$ on $G$ is called a {\it Mackey functor} if it satisfies $M(X)\in\Ob(\Ab)$ for any $X\in\Ob(\Gs)$. The full subcategory of Mackey functors is denoted by $\MackG\subseteq\SMackG$. For the properties of Mackey functors, see for example \cite{Bouc}.
\end{dfn}

Trivial example is the following.
\begin{ex}\label{ExTrivLoc}
Let $M$ be a semi-Mackey functor on $G$. If we define $M^{\times}\subseteq M$ by
\[ M^{\times}(X)=(M(X))^{\times}=\{ \text{invertible elements in}\ M(X) \} \]
for each $X\in\Ob(\Gs)$, then $M^{\times}\subseteq M$ becomes a semi-Mackey subfunctor.
\end{ex}

\begin{prop}\label{PropSemiMackLoc}
Let $\Sl\subseteq M$ be a semi-Mackey subfunctor.
\begin{enumerate}
\item $\Sl^{-1}M=\{ \Sl(X)^{-1}M(X) \}_{X\in\Ob(\Gs)}$ has a structure of a semi-Mackey functor induced from that on $M$. Here, $\Sl(X)^{-1}M(X)$ denotes the ordinary fraction of monoids.
\item The natural monoid homomorphisms
\[ \ell_{\Sl,X}\colon M(X)\rightarrow\Sl^{-1}M(X)\ ;\ x\mapsto \frac{x}{1}\quad({}^{\forall}X\in\Ob(\Gs)) \]
form a morphism of semi-Mackey functors $\ell_{\Sl}\colon M\rightarrow\Sl^{-1}M$.
\item For any semi-Mackey functor $M\ppr$, the above $\ell_{\Sl}$ gives a bijection between the morphisms $\Sl^{-1}M\rightarrow M\ppr$ and the morphisms $\vartheta\colon M\rightarrow M\ppr$ satisfying $\vartheta(\Sl)\subseteq M^{\prime\times}$ $:$
\begin{eqnarray*}
\{ \vartheta\in\SMackG(M,M\ppr)\mid \vartheta(\Sl)\subseteq M^{\prime\times}\}&\overset{\cong}{\longrightarrow}&\SMackG(\Sl^{-1}M,M\ppr)\\
\vartheta&\mapsto&\vartheta\circ\ell_{\Sl}
\end{eqnarray*}
\end{enumerate}
\end{prop}
\begin{proof}
By the universality of the fraction of monoids, for any $f\in\Gs(X,Y)$, there exists a unique monoid homomorphism
\[ f^{\ast}\colon\ \Sl^{-1}M(Y)\rightarrow\Sl^{-1}M(X) \]
compatible with $f^{\ast}$ for $M$, given by
\[ f^{\ast}(\frac{y}{t})=\frac{f^{\ast}(y)}{f^{\ast}(t)}\quad\ ({}^{\forall}\,\frac{y}{t}\in\Sl^{-1}M(Y)).  \]
Similarly $f_{\ast}$ for $\Sl^{-1}M$ is obtained uniquely by
\[ f_{\ast}(\frac{x}{s})=\frac{f_{\ast}(x)}{f_{\ast}(s)}\quad\ ({}^{\forall}\,\frac{x}{s}\in\Sl^{-1}M(X)),  \]
compatibly with $f_{\ast}$ for $M$.
Obviously $\Sl^{-1}M$ becomes a semi-Mackey functor, with these structure morphisms.

The rest also immediately follows from the properties of ordinary fraction of monoids. Since we discuss this again for Tambara functors in Proposition \ref{PropLocUniv}, we omit the detail here. We remark that analogs of Corollary \ref{CorLocLoc} and Corollary \ref{CorLLoc} also hold, which will be left to the reader.
\end{proof}

In particular, we can take the fraction $M^{-1}M$ of a semi-Mackey functor $M$ by itself. This can be understood in a more functorial way as follows. 
\begin{rem}\label{RemFtr}
If $F\colon \Mon\rightarrow\Ab$ is a functor preserving products, then from any semi-Mackey functor $M$, we obtain a Mackey functor
\[ F(M)=\{ F(M(X))\}_{X\in\Ob(\Gs)}. \]
This gives a functor, which we also abbreviate to $F$
\[ F\colon\SMackG\rightarrow\MackG. \]
$($Similarly for functors $\Mon\rightarrow\Mon$, $\Ab\rightarrow\Ab$, and $\Ab\rightarrow\Mon$.$)$
\end{rem}

\begin{ex}\label{ExFtr}
The group-completion functor
\[ K_0\colon\Mon\rightarrow\Ab \]
and the functor taking the group of invertible elements
\[ (\ )^{\times}\colon\Mon\rightarrow\Ab \]
yield functors
\begin{eqnarray*}
&K_0\,\colon\,\SMackG\rightarrow\MackG,&\\
&(\ )^{\times}\,\colon\,\SMackG\rightarrow\MackG.&
\end{eqnarray*}

Moreover the adjoint properties of the original functors are enhanced to this Mackey-functorial level. In fact, it can be easily shown that $K_0$ is left adjoint to the inclusion functor $\MackG\hookrightarrow\SMackG$, and $(\ )^{\times}$ is right adjoint to the same functor.

Thus for any pair of semi-Mackey functors $M$ and $M\ppr$, we have a natural isomorphism
\begin{eqnarray*}
\SMackG(K_0(M),M\ppr)&\cong&\MackG(K_0(M),M^{\prime\times})\\
&\cong&\SMackG(M,M^{\prime\times})\\
&=&\{ \vartheta\in\SMackG(M,M\ppr)\mid \vartheta(M)\subseteq M^{\prime\times}\} ,\\
\end{eqnarray*}
which re-creates the adjoint isomorphism in Proposition \ref{PropSemiMackLoc}, in the case of $\Sl=M$.
\end{ex}

\medskip

\begin{dfn}\label{DefSat}
For any semi-Mackey subfunctor $\Sl\subseteq M$, we define its {\it saturation} $\widetilde{\Sl}$ by
\[ \widetilde{\Sl}(X)=(\Sl(X))^{\sim}. \]

$\widetilde{\Sl}\subseteq M$ becomes again a semi-Mackey subfunctor. We say $\Sl$ is {\it saturated} if it satisfies $\Sl=\widetilde{\Sl}$.
\end{dfn}

\begin{rem}\label{RemTrivSLoc}
Let $M$ be a semi-Mackey functor on $G$.
\begin{enumerate}
\item If a semi-Mackey subfunctor $\Sl\subseteq M$ satisfies $\Sl\subseteq M^{\times}$, then $\ell_{\Sl}$ becomes an isomorphism. In particular if $\Sl$ belongs to $\MackG$, then we have $\Sl\subseteq M^{\times}$ and thus $\ell_{\Sl}$ is an isomorphism.

\item For any semi-Mackey subfunctor $\Sl\subseteq M$, we have a natural isomorphism $\Sl^{-1}M\overset{\cong}{\rightarrow}\widetilde{\Sl}^{-1}M$ compatible with $\ell_{\Sl}$ and $\ell_{\widetilde{\Sl}}$.
\end{enumerate}
\end{rem}
\begin{proof}
These can be confirmed on each object $X\in\Ob(\Gs)$. See also Remark \ref{RemTrivLoc}.
\end{proof}

\section{Semi-Mackey subfunctors generated by $S\subseteq M(G/e)$}

In this section, we state the construction of semi-Mackey subfunctors $\Sl\subseteq M$ from a saturated $G$-invariant submonoid $S\subseteq M(G/e)$.

The following proposition is also used critically in the next section.

\begin{prop}\label{PropLoc}
Let $\Sl\subseteq M$ be a semi-Mackey subfunctor. Then, for any $f\in\Gs(X,Y)$, 
\[ \Sl(X)\subseteq (f^{\ast}\Sl(Y))^{\sim} \]
is satisfied. Namely, for any $s\in\Sl(X)$, there exist some $a\in M(X)$ and $\bar{s}\in\Sl(Y)$ satisfying $f^{\ast}(\bar{s})=as$.
Indeed, $\bar{s}$ can be chosen as $\ \bar{s}=f_{\ast}(s)$.
\end{prop}
\begin{proof}
Let
\[
\xy
(-10,6)*+{X\times_YX}="0";
(10,6)*+{X}="2";
(-10,-6)*+{X}="4";
(10,-6)*+{Y}="6";
(0,0)*+{\square}="7";
{\ar^<<<<<{p_2} "0";"2"};
{\ar_{p_1} "0";"4"};
{\ar^{f} "2";"6"};
{\ar_{f} "4";"6"};
\endxy
\]
be a pullback diagram, and let $\Delta\colon X\rightarrow X\underset{Y}{\times}X$ be the diagonal map. If we put
\begin{eqnarray*}
&Z=X\underset{Y}{\times}X\setminus\Delta(X),&\\
&q_1=p_1|_Z,\quad q_2=p_2|_Z,&
\end{eqnarray*}
then
\[
\xy
(-11,6)*+{X\amalg Z}="0";
(11,6)*+{X}="2";
(-11,-6)*+{X}="4";
(11,-6)*+{Y}="6";
(0,0)*+{\square}="7";
{\ar^<<<<<<<{\id_X\cup q_2} "0";"2"};
{\ar_{\id_X\cup q_1} "0";"4"};
{\ar^{f} "2";"6"};
{\ar_{f} "4";"6"};
\endxy
\]
also becomes a pullback diagram. Thus by Mackey condition, we obtain
\[ f^{\ast}f_{\ast}(s)=(\id_X\cup q_1)_{\ast}(\id_X\cup q_2)^{\ast}(s)=s\cdot(q_{1\ast}q_2^{\ast}(s)) \]
for any $s\in M(X)$.

In particular when $s\in\Sl(X)$, if we put $a=q_{1\ast}q_2^{\ast}(s)$ and $\bar{s}=f_{\ast}(s)$, then it follows $f^{\ast}(\bar{s})=as$ and $\bar{s}\in\Sl(Y)$.
\end{proof}

\begin{cor}\label{CorSat}
If $\Sl\subseteq M$ is a saturated semi-Mackey subfunctor, then for any $X\in\Ob(\Gs)$, we have
\[ \Sl(X)=(\pt_X^{\ast}(\Sl(G/G)))^{\sim}, \]
where $\pt_X\colon X\rightarrow G/G$ is the constant map.
Thus $\Sl$ is determined by $\Sl(G/G)$.
\end{cor}
\begin{proof}
This immediately follows from $\pt_X^{\ast}(\Sl(G/G))\subseteq\Sl(X)\subseteq(\pt_X^{\ast}(\Sl(G/G)))^{\sim}$ and $\Sl(X)=(\Sl(X))^{\sim}$.
\end{proof}

\begin{rem}\label{Rem3Cond}
Let $M$ be a semi-Mackey functor on $G$.
To give a semi-Mackey subfunctor $\Sl\subseteq M$ is equivalent to give a submonoid $\Sl(X)\subseteq M(X)$ for each transitive $X\in\Ob(\Gs)$, in such a way that
\begin{enumerate}
\item[{\rm (i)}] $f_{\ast}(\Sl(X))\subseteq\Sl(Y)$
\item[{\rm (ii)}] $f^{\ast}(\Sl(Y))\subseteq\Sl(X)$
\end{enumerate}
are satisfied for any $f\in\Gs(X,Y)$ between transitive $X,Y\in\Ob(\Gs)$. 

In fact, if we define $\Sl(X)$ for any (not necessarily transitive) $X\in\Ob(\Gs)$ by
\[ \Sl(X)=\{ (s_1,\ldots,s_n)\in M(X)\mid s_i\in\Sl(X_i)\ \ (1\le {}^{\forall}i\le n) \} \]
using the orbit decomposition $X=X_1\amalg \cdots \amalg X_n$, then $\Sl\subseteq M$ becomes a semi-Mackey subfunctor.
\end{rem}

Starting from a $G$-invariant submonoid $S\subseteq M(G/e)$, we can construct semi-Mackey subfunctors of $M$ in the following ways.
\begin{prop}\label{PropMinLoc}
Let $S\subseteq M(G/e)$ be a saturated $G$-invariant submonoid. For each transitive $X\in\Ob(\Gs)$, define $\Ll_S(X)$ by
\[ \Ll_S(X)=\gamma_{X\ast}(S) \]
for some $\gamma_X\in\Gs(G/e,X)$. Then $\Ll_S\subseteq M$ becomes a semi-Mackey subfunctor.

Obviously we have $\Ll_S(G/e)=S$, and $\Ll_S$ is the minimum one among the semi-Mackey subfunctors $\Sl$ satisfying $\Sl(G/e)\supseteq S$.
\end{prop}
\begin{proof}
First remark that the definition of $\Ll_S(X)$ does not depend on the choice of $\gamma_X$, since $S$ is $G$-invariant. We show the conditions in Remark \ref{Rem3Cond} are satisfied.

Let $f\in\Gs(X,Y)$ be any morphism between transitive $X,Y\in\Ob(\Gs)$. Obviously we have $f_{\ast}(\Ll_S(X))=(f\circ\gamma_X)_{\ast}(S)=\Ll_S(Y)$.

For a morphism $\gamma_Y\in\Gs(G/e,Y)$, the fiber product of $f$ and $\gamma_Y$ can be written in the form
\[
\xy
(-12,6)*+{\underset{1\le i\le n}{\amalg}G/e}="0";
(12,6)*+{G/e}="2";
(-12,-6)*+{X}="4";
(12,-6)*+{Y}="6";
(0,0)*+{\square}="7";
(14,-7)*+{,}="9";
{\ar^{\nabla} "0";"2"};
{\ar_<<<<{\underset{1\le i\le n}{\cup}\gamma_i} "0";"4"};
{\ar^{\gamma_Y} "2";"6"};
{\ar_{f} "4";"6"};
\endxy
\]
with some $\gamma_1,\ldots,\gamma_n\in\Gs(G/e,X)$.
Thus for any $s\in S$ we have
\[ f^{\ast}\gamma_{Y\ast}(s)=\prod_{1\le i\le n}\gamma_{i\ast}(s)\ \in\Ll_S(X). \]
Namely, we have $f^{\ast}(\Ll_S(Y))\subseteq\Ll_S(X)$.
\end{proof}

\begin{prop}\label{PropMaxLoc}
Let $S\subseteq M(G/e)$ be a saturated $G$-invariant submonoid. Put $S_0=(\pt_{G/e}^{\ast})^{-1}(S)\subseteq M(G/G)$. For each transitive $X\in\Ob(\Gs)$, define $\Ul_S(X)$ by
\[ \Ul_S(X)=(\, (\pt_X)^{\ast}(S_0)\, )^{\sim}. \]
Then $\Ul_S\subseteq M$ becomes a semi-Mackey subfunctor.
Obviously we have $\Ul_S(G/e)\subseteq S$.
\end{prop}
\begin{proof}
We show the conditions in Remark \ref{Rem3Cond} are satisfied.

Let $f\in\Gs(X,Y)$ be any morphism between transitive $X,Y\in\Ob(\Gs)$.
For any $s\in\Ul_S(Y)$, by definition, there exist $a\in M(Y)$ and $t\in S_0$ such that $as=\pt_Y^{\ast}(t)$ holds.
Then we have
\[ f^{\ast}(a)f^{\ast}(s)=f^{\ast}\pt_Y^{\ast}(t)=\pt_X^{\ast}(t), \]
which means $f^{\ast}(\Ul_S(Y))\subseteq\Ul_S(X)$.

It remains to show {\rm (i)}. We use the following lemma.
\begin{lem}\label{LemD1}
For any transitive $X\in\Ob(\Gs)$ and any $t\in S_0$, we have
\[ (\pt_X)_{\ast}\pt_X^{\ast}(t)\in S_0. \]
\end{lem}
\begin{proof}
Remark that $\pt_X^{\ast}(t)$ is $G$-invariant. If we take a pull-back diagram
\[
\xy
(-10,6)*+{\underset{n}{\amalg}\, G/e}="0";
(10,6)*+{X}="2";
(-10,-6)*+{G/e}="4";
(10,-6)*+{G/G}="6";
(0,0)*+{\square}="7";
(15,-7)*+{,}="9";
{\ar^{{}^{\exists}\zeta} "0";"2"};
{\ar_{\nabla} "0";"4"};
{\ar^{\pt_X} "2";"6"};
{\ar_{\pt_{G/e}} "4";"6"};
\endxy
\]
then we have $\pt_{G/e}^{\ast}(\pt_X)_{\ast}\pt_X^{\ast}(t)=\nabla_{\ast}\zeta^{\ast}\pt_X^{\ast}(t)=(\pt_{G/e}^{\ast}(t))^n\ \in S$.
\end{proof}

For any $s\in\Ul_S(X)$, by definition, there exist $a\in M(X)$ and $t\in S_0$ such that $as=\pt_X^{\ast}(t)$. Thus we have $f_{\ast}(a)f_{\ast}(s)=f_{\ast}\pt_X^{\ast}(t)$.
By Proposition \ref{PropLoc}, there exists $b\in M(Y)$ satisfying
\[ b\cdot f_{\ast}\pt_X^{\ast}(t)=\pt_Y^{\ast}(\pt_Y)_{\ast}f_{\ast}\pt_X^{\ast}(t)=\pt_Y^{\ast}(\pt_X)_{\ast}\pt_X^{\ast}(t). \]

Thus we obtain
\begin{equation}\label{EqD3}
b\, f_{\ast}(a)f_{\ast}(s)=\pt_Y^{\ast}(\pt_X)_{\ast}\pt_X^{\ast}(t).
\end{equation}
By Lemma \ref{LemD1}, we have $(\pt_X)_{\ast}\pt_X^{\ast}(t)\in S_0$, and thus $(\ref{EqD3})$ implies
\[ f_{\ast}(s)\in (\,\pt_Y^{\ast}(S_0)\,)^{\sim}=\Ul_S(Y), \]
and condition {\rm (i)} follows.
\end{proof}

\begin{prop}\label{PropSMax}
Let $S\subseteq M(G/e)$ be a saturated $G$-invariant submonoid. Then $\Ul_S$ is the maximum one among semi-Mackey subfunctors $\Sl$ satisfying $\Sl(G/e)\subseteq S$.
\end{prop}
\begin{proof}
Let $\Sl$ be any semi-Mackey subfunctor satisfying $\Sl(G/e)\subseteq S$.
We have $\Sl\subseteq\widetilde{\Sl}$. Since $S$ is saturated, $\widetilde{\Sl}$ also satisfies $\widetilde{\Sl}(G/e)\subseteq S$.
Since $\widetilde{\Sl}\subseteq M$ is a semi-Mackey subfunctor, we have
\[ \pt_{G/e}^{\ast}(\widetilde{\Sl}(G/G))\subseteq \widetilde{\Sl}(G/e)\subseteq S. \]
Thus it follows
\[ \widetilde{\Sl}(G/G)\subseteq(\pt_{G/e}^{\ast})^{-1}(S)\ \,(=S_0). \]
Since $\widetilde{\Sl}$ is saturated, for any $X\in\Ob(\Gs)$ we have
\[ \widetilde{\Sl}(X)=(\,\pt_X^{\ast}(\widetilde{\Sl}(G/G))\,)^{\sim}\subseteq(\,\pt_X^{\ast}(S_0)\,)^{\sim}=\Ul_S(X) \]
by Corollary \ref{CorSat}.
Thus we obtain $\Sl\subseteq\widetilde{\Sl}\subseteq\Ul_S$.
\end{proof}

\begin{cor}
For any saturated $G$-invariant submonoid $S\subseteq M(G/e)$, we have $\Ll_S\subseteq \Ul_S$. In particular we have $\Ul_S(G/e)=S$.

Moreover, for any semi-Mackey subfunctor $\Sl\subseteq M$ satisfying $\Sl(G/e)=S$, we have $\Ll_S\subseteq \Sl\subseteq \Ul_S$.
\end{cor}
\begin{proof}
This follows from Proposition \ref{PropMinLoc} and Proposition \ref{PropSMax}.
\end{proof}

\section{Fraction of a Tambara functor}

\begin{dfn}\label{DefTamFtr}
A {\it Tambara functor} $T$ {\it on} $G$ is a triplet $T=(T^{\ast},T_+,T_{\bullet})$ of two covariant functors
\[ T_+\colon\Gs\rightarrow\Sett,\ \ T_{\bullet}\colon\Gs\rightarrow\Sett \]
and one contravariant functor
\[ T^{\ast}\colon\Gs\rightarrow\Sett \]
which satisfies the following. 
\begin{enumerate}
\item $T^{\alpha}=(T^{\ast},T_+)$ is a Mackey functor on $G$.
\item $T^{\mu}=(T^{\ast},T_{\bullet})$ is a semi-Mackey functor on $G$.

\noindent Since $T^{\alpha},T^{\mu}$ are semi-Mackey functors, we have $T^{\ast}(X)=T_+(X)=T_{\bullet}(X)$ for each $X\in\Ob(\Gs)$. We denote this by $T(X)$.
\item (Distributive law)
If we are given an exponential diagram
\[
\xy
(-12,6)*+{X}="0";
(-12,-6)*+{Y}="2";
(0,6)*+{A}="4";
(12,6)*+{Z}="6";
(12,-6)*+{B}="8";
(0,0)*+{exp}="10";
{\ar_{f} "0";"2"};
{\ar_{p} "4";"0"};
{\ar_{\lambda} "6";"4"};
{\ar^{\rho} "6";"8"};
{\ar^{q} "8";"2"};
\endxy
\]
in $\Gs$, then
\[
\xy
(-18,7)*+{T(X)}="0";
(-18,-7)*+{T(Y)}="2";
(0,7)*+{T(A)}="4";
(18,7)*+{T(Z)}="6";
(18,-7)*+{T(B)}="8";
{\ar_{T_{\bullet}(f)} "0";"2"};
{\ar_{T_+(p)} "4";"0"};
{\ar^{T^{\ast}(\lambda)} "4";"6"};
{\ar^{T_{\bullet}(\rho)} "6";"8"};
{\ar^{T_+(q)} "8";"2"};
{\ar@{}|\circlearrowright "0";"8"};
\endxy
\]
is commutative. For the exponential diagrams, see \cite{Tam}.
\end{enumerate}

If $T=(T^{\ast},T_+,T_{\bullet})$ is a Tambara functor, then $T(X)$ becomes a ring for each $X\in\Ob(\Gs)$, whose additive (resp. multiplicative) structure is induced from that on $T^{\alpha}(X)$ (resp. $T^{\mu}(X)$).
Those $T^{\ast}(f), T_+(f),T_{\bullet}(f)$ for morphisms $f$ in $\Gs$ are called {\it structure morphisms} of $T$. For each $f\in\Gs(X,Y)$,
\begin{itemize}
\item $T^{\ast}(f)\colon T(Y)\rightarrow T(X)$ is a ring homomorphism, called the {\it restriction} along $f$. 
\item $T_+(f)\colon T(X)\rightarrow T(Y)$ is an additive homomorphism, called the {\it additive transfer} along $f$.
\item $T_{\bullet}(f)\colon T(X)\rightarrow T(Y)$ is a multiplicative homomorphism, called the {\it multiplicative transfer} along $f$.
\end{itemize}
$T^{\ast}(f),T_+(f),T_{\bullet}(f)$ are often abbreviated to $f^{\ast},f_+,f_{\bullet}$.

A {\it morphism} of Tambara functors $\varphi\colon T\rightarrow S$ is a family of ring homomorphisms
\[  \varphi=\{\varphi_X\colon T(X)\rightarrow S(X) \}_{X\in\Ob(\Gs)}, \]
natural with respect to all of the contravariant and the covariant parts. We denote the category of Tambara functors by $\TamG$.
\end{dfn}

\begin{ex}\label{ExTamFtr}
$\ \ $
\begin{enumerate}
\item If we define $\Omega$ by
\[ \Omega(X)=K_0(\Gs/X) \]
for each $X\in\Ob(\Gs)$, where the right hand side is the Grothendieck ring of the category of finite $G$-sets over $X$, then $\Omega$ becomes a Tambara functor on $G$. This is called the {\it Burnside Tambara functor} (\cite{Tam} or \cite{N_IdealTam}).
%

\item Let $R$ be a $G$-ring. If we define $\mathcal{P}_R$ by
\[ \mathcal{P}_R(X)=\{ G\text{-maps from}\ X\ \text{to}\ R \} \]
for each $X\in\Ob(\Gs)$, then $\mathcal{P}_R$ becomes a Tambara functor on $G$. This is called the {\it fixed point functor} associated to $R$ (\cite{Tam} or \cite{N_IdealTam}).
\end{enumerate}
\end{ex}

In this section, we construct a fraction of a Tambara functor by a semi-Mackey subfunctor $\Sl\subseteq T^{\mu}$. As in Example \ref{ExTrivLoc}, we have a trivial semi-Mackey subfunctor $(T^{\mu})^{\times}$, which we also denote simply by $T^{\times}$.

\begin{prop}\label{PropTamLoc}
Let $T$ be a Tambara functor on $G$ and let $\Sl\subseteq T^{\mu}$ be a semi-Mackey subfunctor.
Then $\Sl^{-1}T=\{ \Sl(X)^{-1}T(X) \}_{X\in\Ob(\Gs)}$ has a structure of a Tambara functor induced from that on $T$.

Moreover, the natural ring homomorphisms
\[ \ell_{\Sl,X}\colon T(X)\rightarrow\Sl^{-1}T(X)\ ;\ x\mapsto \frac{x}{1}\quad({}^{\forall}X\in\Ob(\Gs)) \]
form a morphism of Tambara functors $\ell_{\Sl}\colon T\rightarrow\Sl^{-1}T$.
\end{prop}
\begin{proof}
As shown in Proposition \ref{PropSemiMackLoc}, $\Sl^{-1}T^{\mu}$ has a structure of a semi-Mackey functor, with structure morphisms defined by
\begin{eqnarray*}
&f^{\ast}(\frac{y}{t})=\frac{f^{\ast}(y)}{f^{\ast}(t)}\quad\ ({}^{\forall}\,\frac{y}{t}\in\Sl^{-1}T(Y)),& \\
&f_{\bullet}(\frac{x}{s})=\frac{f_{\bullet}(x)}{f_{\bullet}(s)}\quad\ ({}^{\forall}\,\frac{x}{s}\in\Sl^{-1}T(X)),&
\end{eqnarray*}
for each $f\in\Gs(X,Y)$.

Thus it suffices to give additive transfers for $\Sl^{-1}T$, compatibly with the structure on $\Sl^{-1}T^{\mu}$.
Let $f\in\Gs(X,Y)$ be any morphism.

Let $\frac{x}{s}\in\Sl^{-1}T(X)$ be any element. If we put $\bar{s}=f_{\bullet}(s)$, then by Proposition \ref{PropLoc}, we have $f^{\ast}(\bar{s})=as$ for some $a\in T(X)$.
We define the additive transfer of $\Sl^{-1}T$ along $f$ by
\begin{equation}
\label{EqAdd}
f_+\colon \Sl^{-1}T(X)\rightarrow\Sl^{-1}T(Y)\ ;\ \frac{x}{s}\mapsto \frac{f_+(ax)}{\bar{s}}.
\end{equation}

To show the well-definedness, suppose we have $\frac{x}{s}=\frac{x\ppr}{s\ppr}$ in $\Sl^{-1}T(X)$. Namely, there exists $t\in\Sl(X)$ such that $ts\ppr x=tsx\ppr$. Let $a,a\ppr,b\in T(X)$ and $\bar{s},\bar{s}\ppr,\bar{t}\in\Sl(Y)$ be elements satisfying
\begin{equation}
f^{\ast}(\bar{s})=as,\ \ f^{\ast}(\bar{s}\ppr)=a\ppr s\ppr.
\label{Eq_ss}
\end{equation}
and
\[ f^{\ast}(\bar{t})=bt. \]
Then, by the projection formula, we have
\begin{eqnarray*}
&\bar{t}\bar{s}\ppr f_+(ax)=f_+(axf^{\ast}(\bar{t}\bar{s}\ppr))=f_+(axbta\ppr s\ppr)=f_+(aa\ppr bts\ppr x),&\\
&\bar{t}\bar{s}f_+(a\ppr x\ppr)=f_+(a\ppr x\ppr f^{\ast}(\bar{t}\bar{s}))=f_+(a\ppr x\ppr btas)=f_+(aa\ppr bts\ppr x).&
\end{eqnarray*}
This means we have $\frac{f_+(ax)}{\bar{s}}=\frac{f_+(a\ppr x\ppr)}{\bar{s}\ppr}$ in $\Sl^{-1}T(Y)$, and $f_+$ is well-defined. Also, this argument shows that we can use arbitrary $a\in T(X)$ and $\bar{s}\in\Sl(Y)$ instead of $f_{\bullet}(s)$ to define $f_+(\frac{x}{s})$ by $(\ref{EqAdd})$, as long as they satisfy $f^{\ast}(\bar{s})=as$.

To show the additivity of $f_+$, let $\frac{x}{s}$ and $\frac{x\ppr}{s\ppr}$ be arbitrary elements in $\Sl^{-1}T(X)$, and take $a,a\ppr, \bar{s},\bar{s}\ppr$ satisfying $(\ref{Eq_ss})$.
Then we have $f^{\ast}(\bar{s}\bar{s}\ppr)=aa\ppr ss\ppr$, and thus
\[ f_+(\frac{x}{s}+\frac{x\ppr}{s\ppr})=f_+(\frac{s\ppr x+sx\ppr}{ss\ppr})=\frac{f_+(aa\ppr(s\ppr x+sx\ppr))}{\bar{s}\bar{s}\ppr}. \]
On the other hand, we have
\[ f_+(\frac{x}{s})+f_+(\frac{x\ppr}{s\ppr})=\frac{f_+(ax)}{\bar{s}}+\frac{f_+(a\ppr x\ppr)}{\bar{s}\ppr}=\frac{\bar{s}\ppr f_+(ax)+\bar{s}f_+(a\ppr x\ppr)}{\bar{s}\bar{s}\ppr}. \]
By the projection formula, we have
\begin{eqnarray*}
\bar{s}\ppr f_+(ax)+\bar{s}f_+(a\ppr x\ppr)&=&f_+(axf^{\ast}(\bar{s}\ppr))+f_+(a\ppr x\ppr f^{\ast}(\bar{s}))\\
&=&f_+(aa\ppr(s\ppr x+sx\ppr)),
\end{eqnarray*}
and thus
\[ f_+(\frac{x}{s}+\frac{x\ppr}{s\ppr})=f_+(\frac{x}{s})+f_+(\frac{x\ppr}{s\ppr}). \]

With these definitions, we can easily confirm $\ell_{\Sl,X}\circ f^{\ast}=f^{\ast}\circ\ell_{\Sl,Y}$, $\ell_{\Sl,Y}\circ f_{\bullet}=f_{\bullet}\circ\ell_{\Sl,X}$ and $\ell_{\Sl,Y}\circ f_+=f_+\circ\ell_{\Sl,X}$ for each $f\in\Gs(X,Y)$. 

It remains to show the compatibilities between these (to-be) structure morphisms.
\begin{enumerate}
\item[{\rm (i)}] {\bf (functoriality of $(\,)_+$)}

Let $X\overset{f}{\rightarrow}Y\overset{g}{\rightarrow}Z$ be a sequence of morphisms in $\Gs$. For any $\frac{x}{s}\in\Sl^{-1}T(X)$, there exist $a\in T(X)$ and $b\in T(Y)$ satisfying
\begin{eqnarray}
\label{Eq_+1} f^{\ast}f_{\bullet}(s)&=&as,\\
\label{Eq_+2} g^{\ast}g_{\bullet}(f_{\bullet}(s))&=&bf_{\bullet}(s),
\end{eqnarray}
by Proposition \ref{PropLoc}. Thus we have
\[ g_+f_+(\frac{x}{s})=g_+(\frac{f_+(ax)}{f_{\bullet}(s)})=\frac{g_+(b\, f_+(ax))}{g_{\bullet}f_{\bullet}(s)}. \]
On the other hand by $(\ref{Eq_+1})$ and $(\ref{Eq_+2})$, we have
\[ (g\circ f)^{\ast}(g\circ f)_{\bullet}(s)=f^{\ast}(b)\, f^{\ast}f_{\bullet}(s)=f^{\ast}(b) as, \]
and thus
\[ (g\circ f)_+(\frac{x}{s})=\frac{(g\circ f)_+(f^{\ast}(b)ax)}{(g\circ f)_{\bullet}(s)}=\frac{g_+(b\, f_+(ax))}{g_{\bullet}f_{\bullet}(s)}. \]

\item[{\rm (ii)}] ({\bf Mackey condition for $(\Sl^{-1}T)^{\alpha}$})

Let
\[
\xy
(-7,6)*+{X\ppr}="0";
(7,6)*+{X}="2";
(-7,-6)*+{Y\ppr}="4";
(7,-6)*+{Y}="6";
(0,0)*+{\square}="7";
{\ar^{\xi} "0";"2"};
{\ar_{f\ppr} "0";"4"};
{\ar^{f} "2";"6"};
{\ar_{\eta} "4";"6"};
\endxy
\]
be any pullback diagram in $\Gs$.
For any $\frac{y}{t}\in\Sl^{-1}T(Y\ppr)$, there exists $b\ppr\in T(Y\ppr)$ satisfying
\begin{equation}
\label{Eq_+3}
\eta^{\ast}\eta_{\bullet}(t)=b\ppr t,
\end{equation}
and
\[ f^{\ast}\eta_{\bullet}(\frac{y}{t})=f^{\ast}(\frac{\eta_+(b\ppr y)}{\eta_{\bullet}(t)})=\frac{f^{\ast}\eta_+(b\ppr y)}{f^{\ast}\eta_{\bullet}(t)}=\frac{\xi_+f^{\prime\ast}(b\ppr y)}{\xi_{\bullet}f^{\prime\ast}(t)}. \]
On the other hand by $(\ref{Eq_+3})$, we have
\[ \xi^{\ast}\xi_{\bullet}f^{\prime\ast}(t)=\xi^{\ast}f^{\ast}\eta_{\bullet}(t)=f^{\prime\ast}\eta^{\ast}\eta_{\bullet}(t)=f^{\prime\ast}(b\ppr) f^{\prime\ast}(t), \]
and thus
\[ \xi_+f^{\prime\ast}(\frac{y}{t})=\xi_+(\frac{f^{\prime\ast}(y)}{f^{\prime\ast}(t)})=\frac{\xi_+(f^{\prime\ast}(b\ppr)f^{\prime\ast}(y))}{\xi_{\bullet}f^{\prime\ast}(t)}. \]

\item[{\rm (iii)}] ({\bf distributive law for $\Sl^{-1}T$})

Let
\[
\xy
(-12,6)*+{X}="0";
(-12,-6)*+{Y}="2";
(0,6)*+{A}="4";
(12,6)*+{Z}="6";
(12,-6)*+{B}="8";
(0,0)*+{exp}="10";
{\ar_{f} "0";"2"};
{\ar_{p} "4";"0"};
{\ar_{\lambda} "6";"4"};
{\ar^{\rho} "6";"8"};
{\ar^{q} "8";"2"};
\endxy
\]
be any exponential diagram in $\Gs$. 
For any $\frac{x}{s}\in\Sl^{-1}T(A)$, there exists $a\in T(A)$ satisfying
\begin{equation}
\label{Eq_+4}
p^{\ast}p_{\bullet}(s)=as,
\end{equation}
and
\[ f_{\bullet}p_+(\frac{x}{s})=f_{\bullet}(\frac{p_+(ax)}{p_{\bullet}(s)})=\frac{f_{\bullet}p_+(ax)}{f_{\bullet}p_{\bullet}(s)}=\frac{q_+\rho_{\bullet}\lambda^{\ast}(ax)}{f_{\bullet}p_{\bullet}(s)}. \]
On the other hand, if we put $\bar{s}=f_{\bullet}p_{\bullet}(s)$ and $b=\rho_{\bullet}\lambda^{\ast}(a)$, then by $(\ref{Eq_+4})$, we have
\[ q^{\ast}(\bar{s})=q^{\ast}f_{\bullet}p_{\bullet}(s)=\rho_{\bullet}\lambda^{\ast}p^{\ast}p_{\bullet}(s)=\rho_{\bullet}\lambda^{\ast}(as)=b\,\rho_{\bullet}\lambda^{\ast}(s), \]
and thus
\[ q_+\rho_{\bullet}\lambda^{\ast}(\frac{x}{s})=q_+(\frac{\rho_{\bullet}\lambda^{\ast}(x)}{\rho_{\bullet}\lambda^{\ast}(s)})=\frac{q_+(b\,\rho_{\bullet}\lambda^{\ast}(x))}{\bar{s}}=\frac{q_+(\rho_{\bullet}\lambda^{\ast}(ax))}{f_{\bullet}p_{\bullet}(s)}. \]
\end{enumerate}

Thus $\Sl^{-1}T$ becomes a Tambara functor, and Proposition \ref{PropTamLoc} is shown.
\end{proof}

\begin{rem}\label{RemTrivLoc}
Let $T$ be a Tambara functor on $G$.
\begin{enumerate}
\item If a semi-Mackey subfunctor $\Sl\subseteq T^{\mu}$ satisfies $\Sl\subseteq T^{\times}$, then $\ell_{\Sl}$ becomes an isomorphism of Tambara functors. In particular if $\Sl$ belongs to $\MackG$, then we have $\Sl\subseteq T^{\times}$ and thus $\ell_{\Sl}$ is an isomorphism.

\item For any semi-Mackey subfunctor $\Sl\subseteq T^{\mu}$, we have a natural isomorphism $\Sl^{-1}T\overset{\cong}{\rightarrow}\widetilde{\Sl}^{-1}T$ of Tambara functors compatible with $\ell_{\Sl}$ and $\ell_{\widetilde{\Sl}}$.
\end{enumerate}
\end{rem}
\begin{proof}
These can be confirmed on each object $X\in\Ob(\Gs)$, by the ordinary commutative ring theory.
\end{proof}


Naturally, the morphism $\ell_{\Sl}\colon T\rightarrow\Sl^{-1}T$ satisfies the expected universality.
\begin{prop}\label{PropLocUniv}
Let $\varphi\colon T\rightarrow T\ppr$ be a morphism of Tambara functors, and let $\Sl\subseteq T^{\mu}$, $\Sl\ppr\subseteq T^{\prime\mu}$ be semi-Mackey subfunctors. If $\varphi$ satisfies $\varphi(\Sl)\subseteq\Sl\ppr$, then there exists unique morphism
\[ \widetilde{\varphi}\colon\Sl^{-1}T\rightarrow\Sl^{\prime-1}T\ppr \]
compatible with $\varphi$.
\[
\xy
(-12,6)*+{T}="0";
(12,6)*+{T\ppr}="2";
(-12,-6)*+{\Sl^{-1}T}="4";
(12,-6)*+{\Sl^{\prime-1}T\ppr}="6";
{\ar^{\varphi} "0";"2"};
{\ar_{\ell_{\Sl}} "0";"4"};
{\ar^{\ell_{\Sl\ppr}} "2";"6"};
{\ar_{\widetilde{\varphi}} "4";"6"};
{\ar@{}|{\circlearrowright} "0";"6"};
\endxy
\]
\end{prop}
\begin{proof}
By the ordinary commutative ring theory, there exists a unique ring homomorphism
\[ \widetilde{\varphi}_X\colon\Sl^{-1}T(X)\rightarrow\Sl^{\prime-1}T\ppr(X) \]
for each $X\in\Ob(\Gs)$, satisfying $\widetilde{\varphi}_X\circ\ell_{\Sl,X}=\ell_{\Sl\ppr,X}\circ\varphi_X$. This is given by
\[ \widetilde{\varphi}_X(\frac{x}{s})=\frac{\varphi_X(x)}{\varphi_X(s)} \]
for any $\frac{x}{s}\in\Sl^{-1}T(X)$.

It suffices to show $\widetilde{\varphi}=\{\widetilde{\varphi}_X\}_{X\in\Ob(\Gs)}$ is compatible with $f^{\ast},f_+,f_{\bullet}$ for any morphism $f\in\Gs(X,Y)$. 
Compatibility with $f^{\ast}$ and $f_{\bullet}$ immediately follows from the definitions.

We show the compatibility with $f_+$.
For any $\frac{x}{s}\in\Sl^{-1}T(X)$, there exists $a\in T(X)$ satisfying $f^{\ast}f_{\bullet}(s)=as$.
It follows
\[ f^{\ast}f_{\bullet}\varphi_X(s)=\varphi_Xf^{\ast}f_{\bullet}(s)=\varphi_X(a)\varphi_X(s), \]
and thus we obtain
\begin{eqnarray*}
\widetilde{\varphi}_Yf_+(\frac{x}{s})&=&\widetilde{\varphi}_Y(\frac{f_+(ax)}{f_{\bullet}(s)})\ =\ \frac{\varphi_Yf_+(ax)}{\varphi_Yf_{\bullet}(s)}\\
&=&\frac{f_+\varphi_X(ax)}{f_{\bullet}\varphi_X(s)}\ =\ f_+(\frac{\varphi_X(x)}{\varphi_X(s)})\ =\ f_+\widetilde{\varphi}_X(\frac{x}{s}).
\end{eqnarray*}
\[
\xy
(-20,6)*+{\Sl^{-1}T(X)}="0";
(20,6)*+{\Sl^{\prime-1}T\ppr(X)}="2";
(-20,-6)*+{\Sl^{-1}T(Y)}="4";
(20,-6)*+{\Sl^{\prime-1}T\ppr(Y)}="6";
{\ar^{\widetilde{\varphi}_X} "0";"2"};
{\ar_{f_+} "0";"4"};
{\ar^{f_+} "2";"6"};
{\ar_{\widetilde{\varphi}_Y} "4";"6"};
{\ar@{}|{\circlearrowright} "0";"6"};
\endxy
\]
\end{proof}

\begin{cor}\label{CorLocUniv}
Let $T$ be a Tambara functor, and let $\Sl\subseteq T^{\mu}$ be a semi-Mackey subfunctor. Then $\ell_{\Sl}$ gives a bijection between the morphisms $\Sl^{-1}T\rightarrow T\ppr$ and the morphisms $\varphi\colon T\rightarrow T\ppr$ satisfying $\varphi(\Sl)\subseteq T^{\prime\times}$$:$
\begin{eqnarray*}
-\circ \ell_{\Sl}\ \colon\ \{ \varphi\in\TamG(T,T\ppr)\mid \varphi(\Sl)\subseteq T^{\prime\times}\}&\overset{\cong}{\longrightarrow}&\TamG(\Sl^{-1}T,T\ppr)\\
\varphi&\mapsto&\varphi\circ\ell_{\Sl}
\end{eqnarray*}
\end{cor}
\begin{proof}
This immediately follows from Remark \ref{RemTrivLoc} and Proposition \ref{PropLocUniv}.
\end{proof}

Once Proposition \ref{PropLoc} is shown, some natural compatibilities immediately follows from Proposition \ref{PropLocUniv}.

\begin{cor}\label{CorLocLoc}
Let $\varphi\colon T\rightarrow T\ppr$ be a morphism of Tambara functors.
\begin{enumerate}
\item If $\Sl\subseteq T^{\mu}$ is a semi-Mackey subfunctor, then
\[ \varphi(\Sl)=\{ \varphi_X(\Sl(X))\}_{X\in\Ob(\Gs)} \]
gives a semi-Mackey subfunctor $\varphi(\Sl)\subseteq T^{\prime\mu}$, and we obtain a morphism $\Sl^{-1}T\rightarrow (\varphi(\Sl))^{-1}T\ppr$ compatible with $\varphi$.
\item If $\Sl\ppr\subseteq T^{\prime\mu}$ is a semi-Mackey subfunctor, then
\[ \varphi^{-1}(\Sl\ppr)=\{\varphi_X^{-1}(\Sl\ppr(X))\}_{X\in\Ob(\Gs)} \]
gives a semi-Mackey subfunctor $\varphi^{-1}(\Sl\ppr)\subseteq T$, and we obtain a morphism $(\varphi^{-1}(\Sl\ppr))^{-1}T\rightarrow \Sl^{\prime-1}T\ppr$ compatible with $\varphi$.
\end{enumerate}
\end{cor}
\begin{proof}
This immediately follows from Proposition \ref{PropLocUniv}.
\end{proof}

\begin{cor}\label{CorLLoc}
Let $T$ be a Tambara functor, and let $\Sl\subseteq\Sl\ppr\subseteq T$ be semi-Mackey subfunctors. Then the image $\bar{\Sl\ppr}=\ell_{\Sl}(\Sl\ppr)$ of $\Sl\ppr$ under the morphism $\ell_{\Sl}\colon T\rightarrow \Sl^{-1}T$ becomes a semi-Mackey subfunctor $\bar{\Sl\ppr}\subseteq \Sl^{-1}T$, and there is a natural isomorphism
\[ \Sl^{\prime -1}T\overset{\cong}{\longrightarrow}\bar{\Sl}^{-1}(\Sl^{-1}T) \]
compatible with $\ell_{\Sl\ppr}$ and $\ell_{\bar{\Sl}}\circ\ell_{\Sl}$.
\end{cor}
\begin{proof}
This immediately follows from Corollary \ref{CorLocLoc} and an objectwise argument from ordinary commutative ring theory.
\end{proof}

\section{Compatibility with the Tambarization}

In \cite{N_TamMack}, we constructed a functor ({\it Tambarization})
\[ \mathcal{T}\colon\SMackG\rightarrow\TamG, \]
which is left adjoint to the functor taking multiplicative parts
\[ (-)^{\mu}\colon\TamG\rightarrow\SMackG. \]

$\mathcal{T}$ is regarded as a $G$-bivariant analog of the monoid-ring functor
\[ \Mon\rightarrow\Ring\ ;\ Q\mapsto\mathbb{Z}[Q]. \]
In this view, we denote $\mathcal{T}(M)$ by $\Omega [M]$ for any $M\in\Ob(\SMackG)$.

For each $X\in\Ob(\Gs)$, by definition $(\Omega [M])(X)$ is the Grothendieck ring
\[ \Omega [M](X)=K_0(\MGs/X), \]
where $\MGs/X$ is the category defined as follows.
\begin{itemize}
\item[-] An object in $\MGs/X$ is a pair $(A\overset{p}{\rightarrow}X,m)$ of $(A\overset{p}{\rightarrow}X)\in\Ob(\Gs/X)$ and $m\in M(A)$.
\item[-] A morphism from $(A_1\overset{p_1}{\rightarrow}X,m_1)$ to $(A_2\overset{p_2}{\rightarrow}X,m_2)$ is a morphism $f\in\Gs(A_1,A_2)$ satisfying $p_2\circ f=p_1$ and $M^{\ast}(f)(m_2)=m_1$.
\item[-] Sum of $(A_1\overset{p_1}{\rightarrow}X,m_1)$ and $(A_2\overset{p_2}{\rightarrow}X,m_2)$ is
\[ (A_1\amalg A_2\overset{p_1\cup p_2}{\longrightarrow}X,m_1\amalg m_2), \]
where $m_1\amalg m_2$ is the element in $M(A_1\amalg A_2)$ corresponding to $(m_1,m_2)$ under the natural isomorphism $M(A_1\amalg A_2)\cong M(A_1)\times M(A_2)$.
\item[-] Product of $(A_1\overset{p_1}{\rightarrow}X,m_1)$ and $(A_2\overset{p_2}{\rightarrow}X,m_2)$ is $(A\overset{p}{\rightarrow}X,m_1\star m_2)$, where
\[
\xy
(-6,6)*+{A}="0";
(6,6)*+{A_2}="2";
(-6,-6)*+{A_1}="4";
(6,-6)*+{X}="6";
(0,0)*+{\square}="7";
{\ar^{\varpi_2} "0";"2"};
{\ar_{\varpi_1} "0";"4"};
{\ar^{p_2} "2";"6"};
{\ar_{p_1} "4";"6"};
\endxy
\]
is a pullback diagram, and 
\begin{eqnarray*}
&p=p_1\circ\varpi_1=p_2\circ\varpi_2,&\\
&m_1\star m_2=\varpi^{\ast}(m_1)\cdot\varpi_2^{\ast}(m_2).&
\end{eqnarray*}
\end{itemize}

We denote the equivalence class of $(A\overset{p}{\rightarrow}X,m)$ in $\Omega [M](X)$ by $[A\overset{p}{\rightarrow}X,m]$. Any element in $\Omega [M](X)$ can be written in the form of
\[ [A_1\overset{p_1}{\rightarrow}X,m_1]-[A_2\overset{p_2}{\rightarrow}X,m_2] \]
for some $(A_1\overset{p_1}{\rightarrow}X,m_1),\, (A_2\overset{p_2}{\rightarrow}X,m_2)\in\Ob(\MGs/X)$.

\begin{rem}\label{RemHist}
This kind of construction seems to be firstly done by Jacobson in \cite{Jacobson}, and later by Hartmann and Yal\c{c}\i n in \cite{H-R}, to obtain a Green functor from a monoid-valued additive contravariant functor.

Recently this construction was utilized to obtain a Tambara functor from a semi-Mackey functor in \cite{N_TamMack}. This can be also regarded as a generalization of crossed Burnside Tambara functors considered in \cite{O-Y3}.
\end{rem}

For the later use, we briefly recall the construction of the adjoint isomorphism
\begin{eqnarray*}
\TamG(\Omega [M],T)&\cong&\SMackG(M,T^{\mu})\\
\varphi&\leftrightarrow&\vartheta,
\end{eqnarray*}
for each $M\in\Ob(\SMackG)$, $T\in\Ob(\TamG)$. (Theorem 2.15 in \cite{N_TamMack}.)

For any $\varphi\in\TamG(\Omega [M],T)$, the corresponding $\vartheta$ is given by
\begin{eqnarray*}
\vartheta_X\ \colon\ M(X)&\rightarrow&T^{\mu}(X)\\
m&\mapsto&\varphi_X([X\overset{\id_X}{\rightarrow}X,m])
\end{eqnarray*}
for each $X\in\Ob(\Gs)$.

For any $\vartheta\in\SMackG(M,T^{\mu})$, the corresponding $\varphi$ is given by
\begin{eqnarray*}
\varphi_X\ \colon\ \Omega [M](X)&\rightarrow&T(X)\\
{[}A_1\overset{p_1}{\rightarrow}X,m_1{]}-{[}A_2\overset{p_2}{\rightarrow}X,m_2{]}&\mapsto&T_+(p_1)\circ\vartheta_{A_1}(m_1)-T_+(p_2)\circ\vartheta_{A_2}(m_2)
\end{eqnarray*}
for each $X\in\Ob(\Gs)$.

From this, for any semi-Mackey functor $M\in\Ob(\SMackG)$, the adjunction morphism
\[ \varepsilon\colon M\rightarrow\Omega [M]^{\mu} \]
corresponding to $\id_{\Omega [M]}$ is given by
\begin{eqnarray}
\label{EqAdjunction}\varepsilon_X\ \colon\ M(X)&\rightarrow&\Omega [M](X)\\
\nonumber m&\mapsto&[X\overset{\id_X}{\rightarrow}X,m]
\end{eqnarray}
for each $X\in\Ob(\Gs)$.
Remark that, for any $\varphi\in \TamG(\Omega [M],T)$ and corresponding $\vartheta\in\SMackG(M,T^{\mu})$, we have
\begin{equation}\label{EqVarEp}
\varphi\circ\varepsilon=\vartheta.
\end{equation}

By $(\ref{EqAdjunction})$, it is shown that $\varepsilon_X$ is monomorphic for any $X$, and thus $M$ can be regarded as a semi-Mackey subfunctor $M\subseteq\Omega [M]^{\mu}$ through $\varepsilon$. Thus if we are given a semi-Mackey subfunctor $\Sl\subseteq M$, we can localize $\Omega [M]$ by $\varepsilon(\Sl)$.
We denote the fraction $\varepsilon(\Sl)^{-1}(\Omega [M])$ simply by $\Sl^{-1}(\Omega [M])$
. 

\begin{prop}\label{PropLocOmegaM}
Let $M$ be a semi-Mackey functor on $G$, and let $\Sl\subseteq M$ be a semi-Mackey subfunctor. We have a natural isomorphism of Tambara functors
\[ \Sl^{-1}(\Omega [M])\cong \Omega [\Sl^{-1}M]. \]
\end{prop}
\begin{proof}
It suffices to construct a natural bijection
\[ \TamG(\Sl^{-1}(\Omega [M]),T)\cong\TamG(\Omega [\Sl^{-1}M],T) \]
for each $T\in\Ob(\TamG)$.
This is obtained from
\begin{eqnarray*}
\TamG(\Omega [\Sl^{-1}M],T)&\cong&\SMackG(\Sl^{-1}M,T^{\mu})\\
&\cong&\{ \vartheta\in \SMackG(M,T^{\mu})\mid \vartheta(\Sl)\subseteq (T^{\mu})^{\times}=T^{\times} \}
\end{eqnarray*}
and
\begin{eqnarray*}
\TamG(\Sl^{-1}(\Omega [M]),T)&=&\TamG(\varepsilon(\Sl)^{-1}(\Omega [M]),T)\\
&\cong&\{ \varphi\in\TamG(\Omega [M],T)\mid \varphi(\varepsilon(\Sl))\subseteq T^{\times} \} \\
&\cong&\{ \vartheta\in \SMackG(M,T^{\mu})\mid \vartheta(\Sl)\subseteq T^{\times} \}.
\end{eqnarray*}

%
%
\end{proof}

\section{Compatibility with ideal quotients}

In \cite{N_IdealTam}, an ideal of a Tambara functor $T$ was defined as follows.
\begin{dfn}\label{DefIdeal}
Let $T$ be a Tambara functor. An {\it ideal} $\I$ of $T$ is a family of ideals $\I(X)\subseteq T(X)$ $({}^{\forall}X\in\Ob(\Gs))$ satisfying
\begin{enumerate}
\item[{\rm (i)}] $f^{\ast}(\I(Y))\subseteq \I(X)$,
\item[{\rm (ii)}] $f_+(\I(X))\subseteq \I(Y)$,
\item[{\rm (iii)}] $f_{\bullet}(\I(X))\subseteq f_{\bullet}(0)+\I(Y)$
\end{enumerate}
for any $f\in\Gs(X,Y)$. 
\end{dfn}

As shown in \cite{N_IdealTam}, for any ideal $\I\subseteq T$, the quotient
\[ (T/\I)(X)=T(X)/\I(X)\quad(X\in\Ob(\Gs)) \]
forms a Tambara functor $T/\I$, and the projections
\[ p_X\colon T(X)\rightarrow T(X)/\I(X) \quad(X\in\Ob(\Gs)) \]
forms a morphism of Tambara functors $p\colon T\rightarrow T/\I$.

The following gives some examples of ideals (\cite{N_IdealTam}).
\begin{ex}\label{ExIM}
Let $T$ be a Tambara functor, and $I\subseteq T(G/e)$ be a $G$-invariant ideal of $T(G/e)$. 
For each $X\in\Ob(\Gs)$, define $\I_I(X)$ by
\begin{equation}
\label{EqIM}
\I_I(X)=\underset{\gamma\in\Gs(G/e,X)}{\bigcap}(\gamma^{\ast})^{-1}(I).
\end{equation}

Then $\I_I\subseteq T$ becomes an ideal of $T$, which is the maximum one among  ideals $\I$ satisfying $\I(G/e)=I$.
\end{ex}

\begin{rem}\label{RemIS}
Let $T$ be a Tambara functor. Let $I\subseteq T(G/e)$ be a $G$-invariant ideal and let $S\subseteq T(G/e)^{\mu}$ be a saturated $G$-invariant submonoid. 
For any ideal $\I\subseteq T$ satisfying $\I(G/e)=I$ and any semi-Mackey subfunctor $\Sl\subseteq T$ satisfying $\Sl(G/e)=S$, the following are equivalent.

\begin{enumerate}
\item $I\cap S=\emptyset$.
\item $\I\cap \Sl=\emptyset$. Namely, $\I(X)\cap\Sl(X)=\emptyset$ for any non-empty $X\in\Ob(\Gs)$.
\end{enumerate}
\end{rem}
\begin{proof}
Obviously {\rm (2)} implies {\rm (1)}. Conversely, assume {\rm (1)} holds. Then, for any $X\in\Ob(\Gs)$ and $\gamma\in\Gs(G/e,X)$, since
\[ \gamma^{\ast}(\I(X))\subseteq I\quad\text{and}\quad \gamma^{\ast}(\Sl(X))\subseteq S, \]
we obtain
\[ \I(X)\cap\Sl(X)\subseteq (\gamma^{\ast})^{-1}(I\cap S)=\emptyset. \]
\end{proof}

\begin{prop}\label{PropIS}
Let $T$ be a Tambara functor. Let $\I\subseteq T$ be an ideal and $\Sl\subseteq T^{\mu}$ be a semi-Mackey subfunctor, satisfying $\I\cap\Sl=\emptyset$.
\begin{enumerate}
\item If we define $\Sl^{-1}\I\subseteq\Sl^{-1}T$ by
\[ \Sl^{-1}\I(X)=\{ \alpha\in\Sl^{-1}T(X)\mid \alpha=\frac{x}{s}\ \ \text{for some}\ x\in\I(X), s\in\Sl(X) \} \]
for each $X\in\Ob(\Gs)$, then $\Sl^{-1}\I$ becomes an ideal of $\Sl^{-1}T$.
\item Let $p\colon T\rightarrow T/\I$ be the projection, and put $\bar{\Sl}=p(\Sl)$. Then we have a natural isomorphism of Tambara functors
\[ \upsilon\colon\bar{\Sl}^{-1}(T/\I)\overset{\cong}{\longrightarrow}\Sl^{-1}T/\Sl^{-1}\I, \]
compatible with projections.
\begin{equation}
\label{DiagIS}
\xy
(-11,10)*+{T}="0";
(5,10)*+{T/\I}="2";
(-11,-10)*+{\Sl^{-1}T}="4";
(16,8)*+{}="5";
(24,-10)*+{\Sl^{-1}T/\Sl^{-1}\I}="6";
(24,2)*+{\bar{\Sl}^{-1}(T/\I)}="8";
{\ar^{p} "0";"2"};
{\ar_{\ell_{\Sl}} "0";"4"};
{\ar^{\ell_{\bar{\Sl}}} "2";"8"};
{\ar^{} "4";"6"};
{\ar_{\cong}^{\upsilon} "8";"6"};
{\ar@{}|\circlearrowright "4";"5"};
\endxy
\end{equation}
\end{enumerate}
\end{prop}
\begin{proof}
By the ordinary ideal theory for rings, $\Sl^{-1}\I(X)\subseteq \Sl^{-1}T(X)$ becomes an ideal for each $X\in\Ob(\Gs)$. Thus it suffices to show
\[ f^{\ast}(\Sl^{-1}\I(Y))\subseteq \Sl^{-1}\I(X),\quad f_+(\Sl^{-1}\I(X))\subseteq \Sl^{-1}\I(Y), \]
and
\[ f_{\bullet}(\Sl^{-1}\I(X))\subseteq f_{\bullet}(0)+\Sl^{-1}\I(Y). \]

Let $f\in\Gs(X,Y)$ be any morphism. For any $y\in\I(Y)$ and $t\in\Sl(Y)$, we have $f^{\ast}(\frac{y}{t})=\frac{f^{\ast}(y)}{f^{\ast}(t)}\in \Sl^{-1}\I(X)$, and thus
\[ f^{\ast}(\Sl^{-1}\I(Y))\subseteq \Sl^{-1}\I(X). \]
For any $x\in\I(X)$ and $s\in\Sl(X)$, if we take $a\in T(X)$ satisfying $as=f^{\ast}f_{\bullet}(s)$, then we have $f_+(\frac{x}{s})=\frac{f_+(ax)}{f_{\bullet}(s)}\in \Sl^{-1}\I(Y)$, and thus
\[ f_+(\Sl^{-1}\I(X))\subseteq \Sl^{-1}\I(Y). \]
Besides, by $f_{\bullet}(\frac{x}{s})-f_{\bullet}(0)=\frac{f_{\bullet}(x)-f_{\bullet}(s)f_{\bullet}(0)}{f_{\bullet}(s)}=\frac{f_{\bullet}(x)-f_{\bullet}(0)}{f_{\bullet}(s)}\in \Sl^{-1}\I(Y)$, we obtain
\[ f_{\bullet}(\Sl^{-1}\I(X))\subseteq f_{\bullet}(0)+\Sl^{-1}\I(Y) \]
for any $f\in\Gs(X,Y)$.
Thus $\Sl^{-1}\I\subseteq \Sl^{-1}T$ becomes an ideal.

\smallskip

By the ordinary ideal theory for rings, for any $X\in\Ob(\Gs)$, there is a ring isomorphism
\begin{eqnarray*}
\upsilon_X\colon\Sl^{-1}T/\Sl^{-1}\I (X)&\overset{\cong}{\longrightarrow}&\bar{\Sl}^{-1}(T/\I)(X)\\
\frac{x}{s}\ +\Sl^{-1}\I(X)&\mapsto&\frac{p(x)}{p(s)}\qquad({}^{\forall} \frac{x}{s}\in\Sl^{-1}T(X)),
\end{eqnarray*}
which makes $(\ref{DiagIS})$ commutative at $X$.
Since the structure morphisms of $\bar{\Sl}^{-1}(T/\I)$ and $\Sl^{-1}T/\Sl^{-1}\I$ are those induced from $T$, we can check $\upsilon=\{ \upsilon_X\}_{X\in\Ob(\Gs)}$ becomes an isomorphism of Tambara functors.
\end{proof}

\section{Fraction and field-like Tambara functors}

As in \cite{N_IdealTam}, we say a Tambara functor $T$ is {\it field-like} if the zero ideal $(0)\subsetneq T$ is maximal with respect to the inclusion. 
In this section, we consider fractions by the following semi-Mackey subfunctor, and investigate the relations between field-like Tambara functors.

\begin{ex}\label{ExZ}
Let $T$ be a Tambara functor. If we put
\[ \Z=\{ s\in T(G/e)\mid s\ \text{is not a zero divisor}\} , \]
then we obtain two semi-Mackey subfunctors $\Ll_{\Z}\subseteq T^{\mu}$ and $\Ul_{\Z}\subseteq T^{\mu}$.
\end{ex}

We introduce the following condition from \cite{N_IdealTam}.
\begin{dfn}\label{DefMRC}
A Tambara functor $T$ is said to satisfy {\rm (MRC)} if,
for any $f\in\Gs(X,Y)$ between transitive $X,Y\in\Ob(\Gs)$, the restriction $f^{\ast}$ is monomorphic. Remark that we may assume $X=G/e$.
\end{dfn}

\begin{rem}\label{RemMRCQuot}
Let $T$ be a Tambara functor and $\I_{(0)}\subseteq T$ be the ideal corresponding to $(0)\subseteq T(G/e)$ as in Example \ref{ExIM}. If we define $T_{\mathrm{MRC}}$ by $T_{\mathrm{MRC}}=T/\I_{(0)}$, then $T_{\mathrm{MRC}}$ satisfies {\rm (MRC)}. Besides, $T$ satisfies {\rm (MRC)} if and only if $T=T_{\mathrm{MRC}}$.
\end{rem}

\begin{fact}\label{FactMRC}(Theorem 4.21 in \cite{N_IdealTam})
A Tambara functor satisfies {\rm (MRC)} if and only if $T$ is a Tambara subfunctor of $\mathcal{P}_{T(G/e)}$.
\end{fact}

\begin{fact}\label{FactField}(Theorem 4.32 in \cite{N_IdealTam})
For any Tambara functor $T\ne 0$, the following are equivalent.
\begin{enumerate}
\item $T$ is field-like.
\item $T$ satisfies {\rm (MRC)}, and $T(G/e)$ has no non-trivial $G$-invariant ideal.
\end{enumerate}
\end{fact}


First, we show that if $T$ is field-like itself, then nothing is changed under the fraction by $\Ul_{\Z}$. 
\begin{prop}
If $T$ is a field-like Tambara functor, then we have
\[ T^{\times}(G/e)=\Z=\{ s\in T(G/e)\mid \prod_{g\in G}gs\ne0 \}. \]
\end{prop}
\begin{proof}
For any $s\in T(G/e)$, put $\widetilde{s}=\underset{g\in G}{\prod}gs$.
Since we have
\[ T^{\times}(G/e)\subseteq\Z\subseteq\{ s\in T(G/e)\mid \widetilde{s}\ne0 \}, \]
it suffices to show
\[ \{ s\in T(G/e)\mid \widetilde{s}\ne0 \}\subseteq T^{\times}(G/e). \]

Take any $s\in \{ s\in T(G/e)\mid \widetilde{s}\ne0 \}$. Since $\widetilde{s}\ne0$ and $T(G/e)$ contains no non-trivial ideal, we have $\langle\widetilde{s}\rangle=T$. In particular we have $\langle\widetilde{s}\rangle(G/e)\ni 1$.

On the other hand, since $\widetilde{s}$ is $G$-invariant, it can be easily shown that we have
\[ \langle\widetilde{s}\rangle(G/e)=\{ r\ \!\!\widetilde{s}\mid r\in T(G/e) \}. \]
Thus there exists some $r\in T(G/e)$ such that $r\ \!\!\widetilde{s}=1$, which means $\widetilde{s}\in T^{\times}(G/e)$. Consequently we obtain $s\in T^{\times}(G/e)$.
\end{proof}

\begin{prop}\label{PropTrivLoc}
Let $T$ be a field-like Tambara functor. If $S\subseteq T(G/e)$ is a saturated $G$-invariant submonoid contained in $\Z$, then we have
$\Ll_S\subseteq\Ul_S\subseteq T^{\times}$.
\end{prop}
\begin{proof}


Remark that $T$ is a Tambara subfunctor of a fixed point functor. Especially we have $(\pt_{G/e})_{\bullet}\pt_{G/e}^{\ast}(x)=x^{|G|}$ for any $x\in T(G/G)$. Thus if $x\in T(G/G)$ satisfies $\pt_{G/e}^{\ast}(x)\in S\ (\subseteq T^{\times}(G/e))$, then it satisfies $x\in T^{\times}(G/G)$. Namely we have
 \[ (\pt_{G/e}^{\ast})^{-1}(S)\subseteq T^{\times}(G/G). \]  
Thus it follows
\[ \Ul_S(X)\subseteq (\,\pt_{X}^{\ast}(T^{\times}(G/G))\,)^{\sim}\subseteq T^{\times}(X) \]
for any transitive $X\in\Ob(\Gs)$. Thus it follows $\Ul_S\subseteq T^{\times}$.
\end{proof}

\begin{cor}
For any field-like Tambara functor $T$, we have
\[ \Ul_{\Z}^{-1}T\cong\Ll_{\Z}^{-1}T\cong T. \]
\end{cor}
\begin{proof}
This follows from Remark \ref{RemTrivLoc} and Proposition \ref{PropTrivLoc}.
\end{proof}

In the following, we investigate when $\Ul_{\Z}^{-1}T$ becomes field-like.

\begin{rem}\label{RemLocMRC}
Let $T$ be a Tambara functor, and let $\Sl\subseteq T^{\mu}$ be a semi-Mackey subfunctor. Then the following are equivalent.
\begin{enumerate}
\item $\Sl^{-1}T$ satisfies {\rm (MRC)}.
\item For any transitive $X\in\Ob(\Gs)$ and any $x\in T(X)$ admitting some $s\in\Sl(G/e)$ satisfying $s\cdot\gamma_X^{\ast}(x)=0$ for $\gamma_X\in\Gs(G/e,X)$, there exists some $t\in\Sl(X)$ such that $tx=0$. 
\end{enumerate}
Especially, if $\Sl$ satisfies $\Sl(G/e)\subseteq \Z$, then these are also equivalent to $:$
\begin{enumerate}
\item[{\rm (2)$\ppr$}] For any transitive $X\in\Ob(\Gs)$ and any $x\in T(X)$ satisfying $\gamma_X^{\ast}(x)=0$ for $\gamma_X\in\Gs(G/e,X)$, there exists some $t\in\Sl(X)$ such that $tx=0$.
\end{enumerate}
(Condition {\rm (2)} and {\rm (2)$\ppr$} do not depend on the choice of $\gamma_X\in\Gs(G/e,X)$.)
\end{rem}

\begin{prop}\label{PropFinLocMRC}
Let $T$ be a Tambara functor, and let $\Sl\subseteq T^{\mu}$ be a semi-Mackey subfunctor satisfying $\Sl(G/e)\subseteq\Z$. 
Let $p\colon T\rightarrow T/\I_{(0)}=T_{\mathrm{MRC}}$ be the projection, and $\bar{\Sl}\subseteq T_{\mathrm{MRC}}$ be the image of $\Sl$ under $p$. Then we have the following.
\begin{enumerate}
\item $\Sl^{-1}T/\Sl^{-1}\I_{(0)}\cong\bar{\Sl}^{-1}T_{\mathrm{MRC}}$.
\item $\Sl^{-1}T$ satisfies {\rm (MRC)} if and only if the ideal $\Sl^{-1}\I_{(0)}\subseteq\Sl^{-1}T$ is equal to $(0)$.
\item If $\Sl^{-1}T$ satisfies {\rm (MRC)}, then $\Sl^{-1}T\cong\bar{\Sl}^{-1}T_{\mathrm{MRC}}$.
\end{enumerate}
\end{prop}
\begin{proof}
{\rm (1)} follows from Proposition \ref{PropIS}, since $\Sl$ satisfies $\I_{(0)}\cap \Sl=\emptyset$. {\rm (2)} follows from Remark \ref{RemLocMRC}. In fact, for any transitive $X\in\Ob(\Gs)$, the following are equivalent.
\begin{enumerate}
\item[{\rm (i)}] $\Sl^{-1}\I_{(0)}(X)=0$.
\item[{\rm (ii)}] For any $x\in\I_{(0)}(X)$, there exists $t\in \Sl(X)$ satisfying $tx=0$.
\item[{\rm (iii)}] For any $x\in T(X)$ satisfying $\gamma_X^{\ast}(x)=0$ for $\gamma_X\in\Gs(G/e,X)$, there exists $t\in \Sl(X)$ satisfying $tx=0$.
\end{enumerate}
{\rm (3)} follows from {\rm (1)} and {\rm (2)}.
\end{proof}

\begin{lem}\label{LemLocMRC}
Let $T$ be a Tambara functor. If $T$ satisfies one of the following conditions, then $\Ul_{\Z}^{-1}T$ satisfies {\rm (MRC)}.
\begin{enumerate}
\item[{\rm (i)}] $T$ satisfies {\rm (MRC)}.
\item[{\rm (ii)}] For any transitive $X\in\Ob(\Gs)$, if we let $\gamma_X\in\Gs(G/e,X)$ be a $G$-map, then
\[ (\gamma_X)_+(1)\in\Ul_{\Z}(X) \]
holds. $($Remark this does not depend on $\gamma_X$.$)$
\end{enumerate}
\end{lem}
\begin{proof}
We use the criterion of Remark \ref{RemLocMRC}.
\begin{enumerate}
\item[{\rm (i)}] This is obvious, since $\gamma_X^{\ast}(x)=0$ implies $x=0$.
\item[{\rm (ii)}] By the projection formula, $\gamma_X^{\ast}(x)=0$ implies
\[ \gamma_{X+}(1)\cdot x=\gamma_{X+}\gamma_X^{\ast}(x)=0. \]
\end{enumerate}
\end{proof}

As an immediate consequence of Lemma \ref{LemLocMRC}, we have:
\begin{prop}\label{PropLocField}
Let $T$ be a Tambara functor. If $T(G/e)$ is an integral domain and if $T$ satisfies one of the conditions {\rm (i)}, {\rm (ii)} in Lemma \ref{LemLocMRC}, then $\Ul_{\Z}^{-1}T$ becomes a field-like Tambara functor.
\end{prop}
\begin{proof}
By Lemma \ref{LemLocMRC}, $\Ul_{\Z}^{-1}T$ satisfies {\rm (MRC)}. Since $(\Ul_{\Z}^{-1}T)(G/e)$ is a field, $\Ul_{\Z}^{-1}T$ becomes field-like by Fact \ref{FactField}.
\end{proof}

\begin{ex}\label{ExLocMRC}
$\ \ $
\begin{enumerate}
\item For any $G$-ring $R$, the fixed point functor $\mathcal{P}_R$ satisfies condition {\rm (i)} in Lemma \ref{LemLocMRC}. Especially if $R$ is an integral domain, then $\Ul_{\Z}^{-1}\mathcal{P}_R$ becomes a field-like Tambara functor by Proposition \ref{PropLocField}.
\item 
If $T(G/e)$ has no $|G|$-torsion, then Tambara functor $T$ satisfies condition {\rm (ii)} in Lemma \ref{LemLocMRC}.
\end{enumerate}
\end{ex}
\begin{proof}
{\rm (1)} follows immediately from the definition of $\mathcal{P}_R$. We show {\rm (2)}.
Let $X\in\Ob(\Gs)$ be transitive. We may assume $X=G/H$, for some $H\le G$. It suffices to show $(p^H_e)_+(1)\in\Ul_{\Z}(X)$.

By the existence of a pullback diagram
\[
\xy
(-10,8)*+{\underset{|G:H|}{\amalg}\, G/e}="0";
(10,8)*+{G/e}="2";
(-10,-8)*+{G/H}="4";
(10,-8)*+{G/G}="6";
(0,0)*+{\square}="7";
(15,-9)*+{,}="9";
{\ar^{{}^{\exists}\zeta} "0";"2"};
{\ar_{\underset{|G:H|}{\cup}p^H_e} "0";"4"};
{\ar^{p^G_e} "2";"6"};
{\ar_{p^G_H} "4";"6"};
\endxy
\]
we have
\begin{equation}
(p^G_H)^{\ast}(p^G_e)_+(1)=|G:H|\cdot(p^H_e)_+(1).
\label{EqForGamma1}
\end{equation}
In particular if $H=e$, we obtain
\begin{equation}
(p^G_e)^{\ast}(p^G_e)_+(1)=|G|\cdot 1.
\label{EqForGamma2}
\end{equation}

Since $T(G/e)$ has no $|G|$-torsion, we have $|G|\cdot 1\in\Z$, and thus $(\ref{EqForGamma2})$ implies
\[ (p^G_e)_+(1)\in((p^G_e)^{\ast})^{-1}(|G|\cdot 1)\subseteq ((p^G_e)^{\ast})^{-1}(\Z), \]
namely
\[ (p^G_e)_+(1)\in(\pt_{G/e}^{\ast})^{-1}(\Z). \]
From $(\ref{EqForGamma1})$, we obtain
\[ |G:H|\cdot(p^H_e)_+(1)=(p^G_H)^{\ast}(p^G_e)_+(1)=\pt_X^{\ast}((p^G_e)_+(1))\in\pt_X^{\ast}(\,(\pt_{G/e}^{\ast})^{-1}(\Z)\,), \]
and thus $\gamma_{X+}(1)\in\Ul_{\Z}(X)$.
\end{proof}

\begin{cor}\label{CorOmegaLocMRC}
Let $\Omega\in\Ob(\TamG)$ be the Burnside Tambara functor. 
Then $\Ul_{\Z}^{-1}\Omega$ 
becomes a field-like Tambara functor.
\end{cor}
%
\begin{proof}
Since $\Omega(G/e)$ is an integral domain with no $|G|$-torsion, 
this immediately follows from Proposition \ref{PropLocField} and Example \ref{ExLocMRC}.
\end{proof}

\begin{caution}
In \cite{N_IdealTam}, we also considered an analogous notion of an integral domain, as a \lq {\it domain-like}' Tambara functor. In \cite{N_IdealTam}, a Tambara functor $T$ is called {\it domain-like} if the zero ideal $(0)\subseteq T$ is {\it prime}. Typical examples of domain-like Tambara functors are $T=\Omega$ and the fixed point functor $T=\mathcal{P}_R$ associated to an integral domain $R$ (with a $G$-action). For these Tambara functor $T$, the associated fraction $\Ul_{\Z}^{-1}T$ becomes field-like as shown in Example \ref{ExLocMRC} and Corollary \ref{CorOmegaLocMRC}. However in general, we will have to assume some more conditions on a domain-like Tambara functor, if we expect $\Ul_{\Z}^{-1}T$ to be field-like.
\end{caution}

\medskip

By using Proposition \ref{PropFinLocMRC}, we can calculate $\Ul_{\Z}^{-1}\Omega$.
First we remark the following.
\begin{rem}\label{RemOmega}
For each $H\le G$, let $\mathcal{O}(H)$ denote a set of representatives of conjugacy classes of subgroups of $H$. 

Then $\Omega(G/H)$ is a free module over
\[ \{ G/K={[}G/K\overset{p^H_K}{\rightarrow}G/H{]} \mid K\in\mathcal{O}(H) \}, \]
where $p^H_K\colon G/K\rightarrow G/H$ is the canonical projection.

Especially, for any transitive $X\cong G/H\in\Ob(\Gs)$, any $\alpha\in\Omega(X)$ can be decomposed uniquely as
\begin{equation}\label{Eqa}
\alpha=\sum_{K\in\mathcal{O}(H)} m_K\, {[}G/K\overset{p^H_K}{\rightarrow}G/H{]}\qquad(m_K\in\mathbb{Z}).
\end{equation}
\end{rem}

\begin{prop}\label{PropOmegaLoc}
We have an isomorphism of Tambara functors
\[ \Ul_{\Z}^{-1}\Omega \cong\mathcal{P}_{\mathbb{Q}}. \]
\end{prop}
\begin{proof}
As in Corollary \ref{CorOmegaLocMRC}, $\Ul_{\Z}^{-1}\Omega$ satisfies {\rm (MRC)}. Thus by Proposition \ref{PropFinLocMRC}, it suffices to show
\[ \bar{\Ul}_{\Z}^{-1}\Omega_{\mathrm{MRC}}\cong\mathcal{P}_{\mathbb{Q}}, \]
where $\bar{\Ul}_{\Z}\subseteq \Omega_{\mathrm{MRC}}$ is the image of $\Ul_{\Z}$ under the projection $\Omega\rightarrow\Omega_{\mathrm{MRC}}$.

As shown in \cite{N_IdealTam}, the family of ring isomorphisms $\{\wp_H\}_{H\le G}$
\begin{eqnarray*}
\wp_H\colon(\Omega/\I_{(0)})(G/H)&\rightarrow&\mathcal{P}_{\mathbb{Z}}(G/H)=\mathbb{Z}\\
\sum_{K\in\mathcal{O}(H)}m_K\, [G/K\overset{p^H_K}{\rightarrow}G/H]&\mapsto&\sum_{K\in\mathcal{O}(H)}m_K\, |H:K|
\end{eqnarray*}
gives an isomorphism of Tambara functors $\wp\colon\Omega/\I_{(0)}\overset{\cong}{\longrightarrow}\mathcal{P}_{\mathbb{Z}}$. 

Remark that, for $m\in\mathbb{Z}$, we have $\wp_H(m\, [G/H\overset{\id}{\rightarrow}G/H])=0$ if and only if $m=0$.
Additionally, since $m\, [G/H\overset{\id}{\rightarrow}G/H]\in\Ul_{\Z}(G/H)$ for any $0\ne m\in\mathbb{Z}$, we have
\[ \wp(\bar{\Ul}_{\Z})(G/H)=\mathbb{Z}\setminus\{ 0\} \]
for any $H\le G$. Thus it follows
\[ \bar{\Ul}_{\Z}^{-1}(\Omega/\I_{(0)})\cong\wp(\bar{\Ul}_{\Z})^{-1}\mathcal{P}_{\mathbb{Z}}\cong\mathcal{P}_{\mathbb{Q}}. \]
\end{proof}

\end{document}